\documentclass[oneside, a4paper,11pt,reqno]{amsart}
\usepackage{bbm}
\usepackage[T1]{fontenc}
\usepackage{amssymb}
\usepackage{mathpazo}
\usepackage{verbatim}
\usepackage{amsmath}
\usepackage{enumitem}
\usepackage{appendix}
\usepackage{color}
\usepackage{dsfont}
\usepackage[pdftex]{graphicx}
\usepackage{mathtools}

\usepackage{amsfonts}
\usepackage{color}
\definecolor{Red}{rgb}{1,0,0}

\newtheorem{theorem}{Theorem}[section]
      \newtheorem{lemma}[theorem]{Lemma}
      \newtheorem{proposition}[theorem]{Proposition}

\begin{document}

\title [A proof of Sznitman's conjecture about ballistic RWRE]{A proof of Sznitman's conjecture about ballistic RWRE}

\author{Enrique Guerra$^\dagger$ and Alejandro F. Ram\'\i rez$^*$}

\email{($*$) aramirez@mat.uc.cl ($\dagger$) eaguerra@mat.uc.cl}
\address{Facultad de Matem\'aticas, Pontificia Universidad Cat\'olica de Chile
}

\thanks{Enrique Guerra has been partially supported by Fondo Nacional de Desarrollo 
  Cient\'\i fico y Tecnol\'ogico  3180255.
  Alejandro F. Ram\'\i rez has been 
	partially supported by Fondo Nacional de Desarrollo Cient\'ifico y Tecnol\'ogico  
	1180259 and Iniciativa Cient\'ifica Milenio.
            }

\keywords{Random walk in random environment,  ballisticity conditions.}
\subjclass[2010]{60K37, 82D30, 82C41.}

\maketitle

\begin{abstract}
We consider a random walk  in a uniformly elliptic i.i.d. random environment in $\mathbb Z^d$ for
$d\ge 2$. It is believed that whenever the random walk is transient in
a given direction it is necessarily ballistic. In order to quantify
the gap which would be needed to prove this equivalence, several
ballisticity conditions have been introduced. In particular, in \cite{Sz01,Sz02}, Sznitman defined  the so called conditions
$(T)$ and
$(T')$. The first one is the requirement that certain unlikely exit
probabilities from a set of slabs decay exponentially fast with their
width $L$. The second one is the requirement that
for all $\gamma\in (0,1)$ condition $(T)_\gamma$ is satisfied, which
in turn is defined as the requirement that
  the decay is like
$e^{-CL^\gamma}$ for some  $C>0$.
In this article we prove a conjecture of Sznitman of 2002 \cite{Sz02},
stating
that $(T)$ and $(T')$ are equivalent. Hence, this  closes the
circle proving the equivalence 
of conditions $(T)$, $(T')$ and $(T)_\gamma$ for some
$\gamma\in (0,1)$ as conjectured  in  \cite{Sz02}, and also
of each of these
ballisticity conditions with the polynomial
condition $(P)_M$ for $M\ge 15d+5$ introduced by Berger, Drewitz
and Ram\'\i rez in \cite{BDR14}. 
\end{abstract}

\section{Introduction}

Random walk in random environment is one of the most fundamental
mathematical models of probability theory. It describes the movement of a particle in a
disordered landscape and its importance stems due to its relevance as a
reliable framework to study phenomena originating from different
sciences, including  its connection to homogenization
theory through the link between the rescaled particle movement and the
rescaling of the appropriate differential operators (see \cite{Z04},
\cite{Sz04} or \cite{DR14}
for more details and further references). Some basic and simple to state questions about it have remained
persistently open. An example is the relationship
between directional transient behavior, where the random walk drifts
away in a certain direction, and ballisticity, where this drifting
happens with a non-vanishing velocity. For a random walk defined
in the hyper-cubic lattice $\mathbb Z^d$ in an environment which satisfies
some
minimal assumptions, which are uniform ellipticity and which is
i.i.d.,
it is expected that if $d\ge 2$, directional transience implies
ballisticity. In order to tackle this question, several intermediate
conditions which are stronger than directional transience, but in some
sense close to it, have been introduced, with the expectation that
they
would measure the gap needed to prove (or disprove), the conjectured
equivalence between directional transcience and
ballisticity.  Essentially, since directional transience in a given
direction $\ell\in\mathbb S^{d-1}$ implies that the exit probability of
the
random walk from a slab perpendicular to $\ell$, centered at the origin
of
width $L$, through its side in the negative region of space (in $\ell$
coordinate),
decays to $0$ as $L\to\infty$, it is natural to define intermediate
conditions,
called {\it ballisticity conditions}, which quantify this decay, with
the
expectation that they would imply ballisticity. Indeed, in \cite{Sz01}
and
\cite{Sz02}, Sznitman defined the so called conditions $(T)$, $(T')$
and $(T)_\gamma$ for $\gamma\in (0,1)$. The first one is defined as
the requirement that the above mentioned decay is exponentially fast
as $L\to\infty$ for an open set of directions, $(T)_\gamma$ is defined as the requirement that the
decay is like $e^{-CL^\gamma}$ for some constant $C>0$, while
$(T')$ is defined as the fulfillment of $(T)_\gamma$ for all
$\gamma\in (0,1)$. It was shown in \cite{Sz01,Sz02} and in \cite{BDR14} that all
of these conditions imply ballisticity and a functional central limit theorem.
Thanks to these results, and to the fact that condition $(T')$ can
be in principle checked case by case, it was established in \cite{Sz03}, that
it is possible to construct perturbations of the simple symmetric
random walk with a very small local drift, which are nevertheless
ballistic. 

 In \cite{Sz02}, Sznitman conjectured that
$(T)$,$(T')$ and $(T)_\gamma$ for any $\gamma\in (0,1)$ are all
equivalent, proving that for $\gamma\in (0.5,1)$, indeed
$(T)_\gamma$ implies $(T')$. Subsequently Drewitz and Ram\'\i rez
in \cite{DR11}, were able to push this equivalence down to
$\gamma\in (\gamma_d,1)$ for some dimension dependent
constant $\gamma_d\in (0.366,0388)$. In \cite{DR12},
the  equivalence between $(T)$ and $(T')$ was established
for dimensions $d\ge 4$, while in \cite{BDR14} the full
equivalence was proved between these conditions for $d\ge 2$
showing also  that both conditions are equivalent
to an effective polynomial condition $(P)_M$, for $M\ge 15d+5$,
where instead of exponential or stretched exponential decay for
the exit probability through the unlikely side of the slabs, one
imposes a polynomial decay of the form $1/L^M$. Nevertheless,
although a strong indication that the conjectured equivalence between conditions 
$(T)$ and $(T')$ was true was given in \cite{GR15}, the proof of
the equivalence 
remained open.

In this article we prove that for uniformly elliptic i.i.d.
environments conditions $(T)$
and $(T')$ are equivalent, closing the circle and hence finishing the proof
of Sznitman's conjecture of \cite{Sz02} that $(T)$, $(T')$
and $(T)_\gamma$ for any $\gamma\in (0,1)$, are all equivalent.

The proof mimics one-dimensional estimates through a coarse
graining method where sites are mapped into growing strips, introducing
crucial controls on atypically small probabilities, to
decouple the behavior of the random walk in overlapping strips.

In the following section we will define the basic notation and
formulate the main result of this article. In Section
\ref{sectionthree},
the proof of this theorem is presented.

\medskip

\section{Notation and results}
Denote by $|\cdot|_1$ and $|\cdot |_2$ the $l_1$ and $l_2$ norms
respectively,  defined
on $\mathbb Z^d$ and let $U:=\{e\in\mathbb Z^d: |e|_1=1\}=\{e_1,-e_1,\ldots,e_d,-e_d\}$.
Define $\mathcal P:=\{p(e), e\in U:p(e)\ge 0,\sum_{e\in U}p(e)=1\}$
and the {\it environmental space} $\Omega:=\mathcal P^{\mathbb Z^d}$.
We will call an element of $\omega=\{\omega(x):x\in\mathbb Z^d\}\in\Omega$ an {\it environment} 
where for each $x\in\mathbb Z^d$,
 $\omega(x)=\{\omega(x,e),e\in U\}\in\mathcal P$.
A random walk in a fixed environment $\omega$ starting
from $x\in\mathbb Z^d$ is defined as
the Markov chain $\{X_n:n\ge 0\}$ with $X_0=x$ and transition
probabilities to jump from a site $y$ to a nearest neighbor $y+e$,
$\omega(y,e)$. We denote by $P_{x,\omega}$ the law of
this random walk. Whenever a probability measure $\mathbb P$ is
prescribed on $\Omega$,  we call $P_{x,\omega}$ the
{\it quenched law} of the random walk in random environment (RWRE).
We define the {\it averaged} or {\it annealed} measure 
of the RWRE as the semidirect product $P_x:=\mathbb P\times P_{x,\omega}$
defined on $\Omega\times (\mathbb Z^d)^{\mathbb N}$.

We will say  that $\mathbb P$ is {\it  uniformly elliptic}
if there is a constant $\kappa>0$ such
that $\mathbb P(\omega(x,e)\ge \kappa)=1$
and that it is i.i.d. if the random variables $\{\omega(x):x\in
\mathbb Z^d\}$ are i.i.d. under $\mathbb P$. Throughout the
rest of this article we will assume that $\mathbb P$ is
uniformly elliptic and i.i.d.

Given a direction $\ell\in\mathbb S^{d-1}$, we say that the
random walk is transient in direction $\ell$ if

$$
\lim_{n\to\infty}X_n\cdot \ell=\infty.
$$
Furthermore, we say that the random walk is ballistic
in direction $\ell$ if

$$
\liminf_{n\to\infty}\frac{1}{n}X_n\cdot \ell>0.
$$
It is believed that for dimensions $d\ge 2$, whenever a random
walk in a uniformly elliptic i.i.d. environment is transient
in a given direction, it is necessarily ballistic \cite{DR14}.
In order to give at least
a partial answer to the above question, several conditions which interpolate between directional
transience and ballisticity, called {\it ballisticity conditions},
have been introduced: conditions $(T)$, $(T')$, $(T)_\gamma$ for
$\gamma\in (0,1)$ defined by Sznitman in \cite{Sz01,Sz02} and $(P)_M$ 
for $M\ge 1$, defined by Berger, Drewitz and Ram\'\i rez in \cite{BDR14}.

For $L\in \mathbb R$ and $\ell\in \mathcal S^{d-1}$ 
define the stopping times

\begin{equation}
\label{stopleftright0}
T_L^\ell:=\inf\{n\geq 0:\,X_n\cdot \ell\geq L \}\hspace{1.5ex}
\end{equation}
and
\begin{equation}
\label{stopleftright}
\tilde T_{L}^\ell:=\inf\{n\geq0 \, X_n\cdot \ell \leq L\}.
\end{equation}
Let $\gamma\in (0,1]$. We say that condition $(T)_\gamma$ in direction $\ell$, also denoted
by $(T)_\gamma|\ell$, is satisfied if there exists an open subset $O$ of
$\mathbb S^{d-1}$ containing $\ell$ such that for all $\ell'\in O$
we have that

\begin{equation}
  \nonumber
  \limsup_{\substack{L\rightarrow\infty}}L^{-\gamma}\log 
P_0\left[\tilde T_{-L}^{\ell'}<T_L^{\ell'}\right]<0.
\end{equation}
 Condition $(T)|\ell$ is defined as $(T)_1|\ell$, while
condition $(T')|\ell$ as the requirement that $(T^\gamma)|\ell$ is fulfilled for all $\gamma\in (0,1)$. For $M\ge 1$, the polynomial condition $(P)_M$,
is essentially defined as the requirement that there exists an $L_0$ such
that for some $L\ge L_0$ we have that

$$
P_0\left[\tilde T_{-L}^{\ell'}<T_L^{\ell'}\right]\le \frac{1}{L^M}
$$ 
 (see \cite{BDR14} for the exact definition). Note the {\it effective}
nature of the polynomial condition as opposed to conditions $(T)$, $(T')$
and $(T)_\gamma$, in the sense that it is a condition that in principle
can be verified for any given uniformly elliptic i.i.d. law $\mathbb P$.

In a series of works \cite{DR11,DR12, BDR14}, culminating with the introduction of the
polynomial condition in \cite{BDR14}, the equivalence between
conditions $(T')$, $(T)_\gamma$ for a given $\gamma\in (0,1)$ and
the polynomial condition $(P)_M$ for $M\ge 15d+5$, was established
for all dimensions $d\ge 2$. Nevertheless, the conjectured
equivalence between $(T)$ and $(T')$ has remained open.
In this article we prove this equivalence.

\medskip

\begin{theorem}
 \label{theoremequivalence} Consider a random walk in an i.i.d.
uniformly elliptic environment satisfying condition $(T')$.
Then, condition $(T)$ is satisfied.
\end{theorem}

\medskip
An automatic corollary of Theorem \ref{theoremequivalence},  is the
extension of the applicability of Yilmaz 
large deviation result of \cite{Y11} stating that under condition 
$(T)$ there is equality between 
the quenched and annealed large deviation rate functions for random walks in uniformly elliptic 
i.i.d. environments in $d\ge 4$, to random walks satisfying 
condition $(P)_M$ for $M\ge 15d+5$ for which condition $(T)$ had not
been proved directly, as the perturbative examples of \cite{Sz03} and
the
more recent ones  in \cite{RS18}.

The proof of Theorem \ref{theoremequivalence}, is based on
a new method which captures the independence of events
defined in overlaping slabs through a carefull use of
atypical quenched exit estimates.
To explain in more detail this new strategy, we recall how  Sznitman in \cite{Sz02} proved
that $(T)_\gamma$ for $\gamma\in (0.5,1)$ implies $(T')$.
He introduced the so called
{\it effective criterion}, which somehow
 mimics the well
known  criteria proved by Solomon in \cite{So75} for random walks
on $\mathbb Z$ in an i.i.d. elliptic environment,
which says that if the expectation of $\omega(0,-1)/(1-\omega(0,-1))$
is smaller than one, the random walk is ballistic to the right.
He then proved that $(T)_\gamma$ for $\gamma\in (0.5,1)$ implies the
effective criterion, and  that the  effective criterion implies
$(T')$. 
The
effective criterion is defined through a quantity $\rho$ analogous to
the one dimensional quotient, but defined
in a larger scale, and basically it is required that some power of
$\rho$ should have a small expectation
at some scale.
As for one-dimensional random walks, somehow $\rho$ can be used
to expand the probability to exit through the left side of large
slabs. To obtain an exponential decay of this probability (which would
hence prove $(T)$), it is necessary to compute the expected value of
products of $\rho_i$'s, where  each $i\in\mathbb Z$ labels a sub-slab of
the large slab, and $\rho_i$ is
distributed as $\rho$. On the other hand the sub-slabs overlap
by pairs, so that for a given $i$, $\rho_i$ and $\rho_{i+1}$
are not independent. This lack of independence is a big complication
to obtain adequate estimates for the expectation, and in order to
prove that the effective criterion implies condition $(T')$, Sznitman
in \cite{Sz02}, decoupled the computation of the expectation of 
products of the $\rho_i$'s using Cauchy-Schwarz inequality.
The iterative use of this argument caused in the end a
decay which was not better than $e^{-L^{\gamma_L}}$ as $L\to\infty$, where $L$ is
the width of the final slab and
$\gamma_L=1-\frac{C}{(\log L)^{1/2}}$, and hence the
proof of $(T')$. Improving this argument through the
use of Cauchy-Schwarz inequality is possible, and produces
a decay 
which is somehow almost exponentially fast (see \cite{GR15}
for details), but still not enough to obtain the exponential decay
of condition $(T)$. A key idea introduced in this article to  prove Theorem \ref{theoremequivalence}
to avoid the use of Cauchy-Schwarz inequality
is to compare the probability to exit through
the left side of a sub-slab, with the probability to exit through
the left  side of the sub-slab without never moving to the right
of the initial departure point. This comparison is done
 through the use 
atypical quenched exit estimates which control the smallness
of the quenched exit probability of the random walk through
atypical exit points. These estimates have been extensively used
in \cite{Sz01,Sz02,DR12,BDR14,FH13} in the i.i.d. case
and more recently in \cite{G17} in the case of environments satisfying
some
kind of mixing condition. Once this comparison
is done in a proper way, the computation of the expected
value of the $\rho_i$'s is reduced to expectations of products of independent
terms, essentially obtaining the desired exponential
decay. These new methods are inspired on Sznitman's effective criterion
but do not rely directly on it.

\medskip
\section{Proof of  Theorem \ref{theoremequivalence}}
\label{sectionthree}
The proof of Theorem \ref{theoremequivalence} will be done obtaining
recursively estimates on the annealed atypical exit probability from
appropriate boxes of the walk in an increasing sequence of scales.
To do this, each box at a given scale will be subdivided in smaller
boxes of a size of the previous scale. A key step here will be  to decouple (in the sense of independence) the atypical exit
probabilities of  overlapping slabs.
To implement our recursive argument, we will need a seed inequality,
which
will enable us to pass from estimates at a given scale to estimates at
the next scale. This is the content of Section \ref{sectionseed}. In Section
\ref{sectionrecursion}, the seed estimate is used to implement the recursion, and
obtain a final estimate for the decay of the atypical exit probability
at
a given scale. In Section \ref{sectionfinal}, this estimate is used to finish the
proof of Theorem \ref{theoremequivalence}.

\subsection{Seed estimate}
\label{sectionseed}
Here we will derive an estimate for the annealed atypical exit
probability of the random walk at a given scale in terms of the
annealed atypical exit probability of the random walk at a smaller
scale. In order to obtain a useful estimate, we will compare
the atypical exit probability of a given box, with the event that
the random walk exits atypically without crossing its starting level
in the direction opposite to the atypical side. This comparison will
be made through the use of classical atypical quenched exit estimates
obtained by Sznitman in \cite{Sz01,Sz02}. Once these probabilities
are compared, it will be basically possible to argue that the quenched
atypical exit probabilities from  overlapping slabs at the smaller scale
are independent.

Let us introduce some notation. Given a subset $V\subset\mathbb Z^d$,
we define its boundary by

$$
\partial V:=\{x\notin V:|x-y|_1=1\ \rm{for}\ \rm{some}\ y\in V\}.
$$
We are assuming that there is a
direction $\ell\in\mathbb S^{d-1}$ such that $(T')|\ell$ is satisfied.
This means that we have a stretched exponential decay through atypical
sides of slabs for $\ell'\in O$, where $O$ is an open subset of
$\mathbb S^{d-1}$ containing $\ell$. Let us fix $\ell'\in O$ and let
$R$ be a rotation on $\mathbb R^d$ defined by $R(e_1)=\ell'$, 
$L>0$ and $\tilde L>0$. To simplify notation, we define the triple
${\bf S}:=(R,L,\tilde L)$ and
define the box associate to the triple $\bf S$ by

\begin{equation}
\label{triple}
B_{\bf S}:=R\left((-L,L)\times (-\tilde L, \tilde L)^{d-1}\right)\cap \mathbb Z^d
\end{equation}
and its {\it positive boundary} or {\it positive side} by

\begin{equation}
  \nonumber
\partial^+ B_{\bf S}:=\partial B_{\bf S}\cap \left\{z:\,
  z\cdot \ell'\geq L,\, |z\cdot R(e_k)|<\tilde L,\ {\rm for}\ 2\le k\in d\right\}.
\end{equation}
We also define the random variable attached to $\bf S$, 
\begin{equation}
  \nonumber
q_{\bf S}:=P_{0,\omega}\left[X_{T_{B_{\bf S}}} \notin\partial^+B_{\bf S} \right]. 
\end{equation}
 Let now
$L_0, \,\tilde L_0,\, L_1$ and $\tilde L_1$ be integers
greater that $3\sqrt d$ and such
that

\begin{equation}
  \nonumber
N:=\frac{L_1}{L_0}\in \mathbb N\cap\{z\in \mathbb R:\, z\geq 2\} \hspace{0.7ex}\mbox{and }\hspace{0.5ex}\tilde N:=\frac{\tilde L_1}{\tilde L_0}\in \{z\in \mathbb N:\, z>N\}. 
\end{equation}
Throughout the rest of this article we will use the fact that there
exits a constant

\begin{equation}
  \label{c1}
  c_1=c_1(d)
\end{equation}
such that given any pair of points
$x,y\in\mathbb Z^d$, there exits a nearest neighbor path of length
at most $c_1|x-y|_1$ joining them. Furthermore, in general the constants
(which  might depend  on the dimension $d$ and the ellipticity
constant $\kappa$) will be denoted by
$c_1(d,\kappa),c_2(d,\kappa),\ldots$, sometimes just writing
 $c_1,c_2,\ldots$.

Now, consider the corresponding triples ${\bf S}_{0}:=(R, L_0, \tilde L_0)$ and ${\bf S}_{1}:=(R, L_1, \tilde L_1)$.
The following seed estimate provides us with an upper bound for
$\mathbb E\left[q_{\bf S_1}\right]$ in terms of $\mathbb E\left[q_{\bf
    S_0}\right]$.

\medskip
\begin{proposition}[Seed estimate]
\label{prop1} 
Let $d\ge 2$. Consider a random walk in an i.i.d. uniformly elliptic
environment. Let $\ell \in \mathbb S^{d-1}$ and assume that condition
$(T')|\ell$ is satisfied
and let $\beta\in (1/2,1)$.
Then there exists  $c_2(d,\kappa)>0$, $\mu>0$ and an open set
$O\subset\mathbb S^{d-1}$ which contains $\ell$, such that for all $\ell'\in O$,
$L_0>3\sqrt d$, $\tilde L_0>3\sqrt d$ and  $\tilde N\geq 48N$ we have
that

\begin{gather}
\nonumber
\mathbb E\left[q_{{\bf S}_1}\right]\leq 
(N+2) 
\left(c_2\kappa^{-3c_1}\tilde L_1^{d-1}e^{3c_1\log(1/\kappa)L_0^\beta}
\right)^{2N+2}\mathbb E\left[q_{{\bf S}_0}\right]^N\\
\label{inepropo}
+
\left(c_2\tilde L_1^{d-2}\frac{L_1^3}{L_0^2}\tilde L_0\mathbb E\left[q_{{\bf S}_0}\right]\right)^{\frac{\tilde N}{12 N}}
+c_2N\tilde L_1^{d-1}e^{-\mu L_0^{d(2\beta-1)}}.
\end{gather}
\end{proposition}
\begin{proof} The strategy will be to divide the box defined by the triple ${\bf 
    S}_1=(R,L_1,\tilde L_1)$ into smaller sub-boxes corresponding to ${\bf 
    S}_0=(R,L_0,\tilde L_0)$. We should have in mind that the
  eventually the length $\tilde L_0$ and $\tilde L_1$ will be chosen
  much larger than $L_0$ and $L_1$ respectively, so  the boxes will
  look very much like slabs. We will divide the proof in eight steps:
  in step 0, we give the necessary definitions to make the above
  described subdivision; in step 1, we use classical one-dimensional
  arguments to estimate the probability that the random walk exits
  through
  the atypical sides of the large box (the back and lateral sides), in
  terms of an expansion involving a quantity analogous to the
  one-dimensional quotient between jumping to the left and to the
  right;
  in step 2 we will define an important typical quenched event which
  will eventually decouple the behavior in different slabs; in step 3
  we will use the above definition to make this decoupling; in steps 4
  and 5,
  we will apply the previous decoupling to express the atypical
  exit probability from the large box, in terms of a product of
  corresponding probabilities in the small scale; in step 6 we bound
  the atypical quenched exit event; in step 7, we bound the
  probability of lateral exit of the random walk; and finally in step
  8 we combine the estimates of the previous steps to finish the
  proof.

  \medskip

  \noindent {\it Step 0: preliminary definitions.}
For each $i\in \mathbb Z$ define the set
\begin{equation}
  \nonumber
  \mathcal H_i:=\{x\in \mathbb Z^d: 
  |x-x'|_1=1\ {\rm and}\ (x\cdot \ell'-iL_0)(x'\cdot\ell'-iL_0)\leq 0\ {\rm for}\
  {\rm some}\ x'\in\mathbb Z^d\}
\end{equation}
as well as the function $I:\mathbb Z^d\rightarrow \mathbb Z$ given by 
\begin{equation}
 \nonumber
I(x)=i, \,\,\mbox{whenever}\,\,x\cdot \ell\in \left[iL_0-\frac{L_0}{2}, iL_0+\frac{L_0}{2}\right). 
\end{equation}
Let $(\theta_n)_{n\ge 0}$ be the canonical shift on $(\mathbb Z^d)^{\mathbb N}$.
We then define the successive visit times  to the different $(\mathcal H_i)_{i\in\mathbb Z}$ sets as 
\begin{gather}
\nonumber 
V_0=0,\,\, V_1=\inf\{n\geq 0: X_n\in \mathcal H_{I(X_0)+1}\cup \mathcal H_{I(X_0)-1}\}\\
\nonumber
\mbox{and }\,\, V_{i+1}=V_1\circ \theta_{V_i}+V_i \,\, \mbox{for $i\geq 1$}. 
\end{gather}
Let us also define the first exit time of the random walk from the box
$B_{{\bf S}_1}$ through its lateral sides as

\begin{equation}
\nonumber
  \tilde T:=\inf\{n\geq 0: |X_n\cdot R(e_j)|\geq \tilde L_1\, \mbox{for some $j\in [2,d]$}\}. 
\end{equation}
In order to mimic the one-dimensional criteria for ballisticity of
random walks in random environment \cite{So75},
it will be convenient to consider for each $x\in \mathbb Z^d$ and integer $i$ the random variables defined by the equations 
\begin{gather}
\nonumber 
 q(x,\omega):=P_{x,\omega}[X_{V_1}\in \mathcal H_{I(x)-1}]=:1- p(x,\omega),\\
\nonumber 
\hat q(x,\omega):=P_{x,\omega}[X_{V_1}\in \mathcal H_{I(x)-1}, V_1\leq \tilde T],\\
\nonumber 
\hat p(x,\omega):=P_{x,\omega}[X_{V_1}\in \mathcal H_{I(x)+1}, V_1\leq
\tilde T]\ {\rm and}\\
\label{defrho}
\hat \rho(i,\omega):=\sup\left\{\frac{\hat q(x,\omega)}{p(x,\omega)}:
  x\in \mathcal H_{i},\,\,|x\cdot R(e_j)|_2<\tilde L_1\ {\rm for}\
  {\rm all}\  2\le j \le d\right\}.
\end{gather}
Consider the
function
$f:\Omega\times \mathbb Z\rightarrow \mathbb R^+$
defined by
\begin{gather}
\nonumber 
f(j,\omega)=0 \hspace{2ex} \mbox{for }\hspace{2ex} j\geq N+2 \quad
{\rm and}\\
\label{defffunc}
f(j,\omega)=\sum_{j \leq m \leq N+1 }\prod_{m<i\leq N+1}\hat \rho (i,\omega)^{-1} \hspace{2ex} \mbox{for }\hspace{2ex} j\leq N+1. 
\end{gather}
Throughout the rest of the proof we may not write explicitly the
dependence on $\omega$ of the random variables involved. 

\medskip

\noindent {\it Step 1: one-dimensional argument to bound exit probabilities.}
Here we will use  one-dimensional explicit formulas for exit
probabilities to obtain the following bound.
\medskip

\begin{equation}
\label{quenine}
P_{0,\omega}\left[\tilde T_{-L_1}^{\ell'}<\tilde T\wedge T_{L_1}^{\ell'}\right]\leq \frac{f(0)}{f(-N)},
\end{equation}
[cf. (\ref{stopleftright}) and (\ref{stopleftright0})].
The proof of (\ref{quenine}) is similar to the proof of inequality
(2.18) of \cite{Sz02}, but for completness we will outline it here.
Consider the $\left(\mathcal F_{V_m}\right)_{m\geq0}$-stopping time  
\begin{equation*}
\tau:=\inf\left\{m\geq0: X_{V_m}\in \mathcal H_{-N}\cup \mathcal H_{N+1}\right\}.
\end{equation*}
We now assert that the random variables 
\begin{equation}
\nonumber
 E_{0,\omega}\left[f(I(X_{V_{m\wedge \tau}})), V_{m\wedge \tau}\leq \tilde T\right]
\end{equation}
are decreasing  with $m$. Indeed, as in (2.20) of (\cite{Sz02})  observe that 
\begin{gather}
\nonumber 
E_{0,\omega}\left[f(I(X_{V_{(m+1)\wedge \tau}})), V_{(m+1)\wedge \tau}\leq \tilde T\right]\\
\nonumber 
\le E_{0,\omega}\left[f(I(X_{V_{m\wedge \tau}})), V_{m\wedge \tau}\leq \tilde T, \tau\leq m \right]+
E_{0,\omega}\left[f(I(X_{V_{m+1}})), V_{m+1}<\tilde T, \tau>m \right]\\ 
\nonumber
= E_{0,\omega}\left[f(I(X_{V_{m \wedge \tau}})), V_{m\wedge \tau}\leq \tilde T, \tau\leq m\right]
+E_{0,\omega}\left[\tau>m, V_m<\tilde T, E_{X_{V_m},\omega}\left[f(I(X_{V_1})), V_1\leq \tilde T\right]\right].
\end{gather}
On the other hand, on $\{\tau>m, V_m<\tilde T\}$
(recall definitions in (\ref{defrho}))
we have that $P_{X_{V_m},\omega}$-a.s

\begin{gather}
\nonumber 
E_{X_{V_m},\omega}\left[f(I(X_{V_1})), V_1\leq \tilde T\right]\\
\nonumber 
\leq f(I(X_{V_m}))+  p(X_{V_m},\omega)\left(f(I(X_{V_m})+1)-f(I(X_{V_m}))\right)\\
\nonumber 
+\hat q(X_{V_m},\omega)\left(f(I(X_{V_m})-1)-f(I(X_{V_m}))\right)\\
\label{fnegt}
\le f(I(X_{V_m}))+ \prod_{I(X_{V_m})-1<j\leq N+1}\hat \rho(j)^{-1}\left(\hat q(X_{V_m}, \omega)- p(X_{V_m},\omega)\hat \rho(I(X_{V_m}))\right),
\end{gather}
where we have used that $\hat p(X_{V_m},\omega)\leq
p(X_{V_m},\omega)$ and 
$p(X_{V_m},\omega)+\hat q(X_{V_m},\omega)\leq 1$ in the first inequality 
and the explicit expression for $\hat \rho$ in 
(\ref{defffunc}) to get the last inequality.
At the same time, by the definitions (\ref{defrho}) and by the fact that
$P_{0,\omega}$-a.s. on the event $\{V_m<\tilde T\}$, we have
$X_{V_m}\in \mathcal H_{I(X_{V_m})}\cap\{z\in \mathbb Z^d:\, |z\cdot
R(e_i)|<\tilde L_1\ {\rm for}\ {\rm all}\ 2\le i\le d \}$  we conclude that the last
term in (\ref{fnegt}) is negative.
As a result, using  Fatou's lemma,  we have 
\begin{equation*}
E_{0,\omega}\left[f(I(X_{V_\tau})), V_\tau\leq \tilde T, \tilde T_{-L_1}^{\ell'}<\tilde T\wedge T_{L_1}^{\ell'}\right]\leq f(0).\\
\end{equation*}
The claim (\ref{quenine}) follows now after
 noticing $P_{0,\omega}$-a.s. on the event appearing in (\ref{quenine}), that
$X_{V_\tau}\in \mathcal H_{-N}$ and $V_\tau\le\tilde T$.

\medskip
\noindent {\it Step 2: typical quenched exit event.}
Here we will define an exit event for the random walk at a given slab,
which corresponds somehow to a minimal size the typical quenched exit
probability should have. 
 Let us start introducing for each $i\in\mathbb Z$, the \textit{frontal} part of the $\mathcal H_i$-boundary,
 \begin{equation}
  \nonumber
\partial^+\mathcal H_i=\partial \mathcal H_i\cap \{z:\,z\cdot \ell'-iL_0\geq 0\}.
\end{equation}
Define also for each $x\in\mathbb Z^d$ the quenched probabilities that
the random walk starting from $x$ exits the corresponding slab through
its atypical side, but without ever visiting the right-hand half of
the slab, as
\begin{equation}\label{qinfty}
\tilde q(x, \omega)= P_{x,\omega}\left[X_{V_1}\in \mathcal H_{I(x)-1}, H_{\partial^+\mathcal H_{I(x)}}=\infty\right].
\end{equation}
The above probability will be a key definition in our proof since it
will be in a sense comparable to $q(x,\omega)$, but it will enable us
to produce enough independence in the products appearing in the
right-hand side of (\ref{quenine}).

 Let us also define the truncation of $\mathcal H_i$ as 

 \begin{gather}
   \nonumber 
\mathcal H'_i:= \mathcal H_i\cap 
\left\{z: |z\cdot R(e_j)|<\tilde L_1,\, \ {\rm for}\ {\rm all}\ 2\le j\le d\right\},
\end{gather}
 its frontal boundary by

 \begin{gather}
   \nonumber 
\partial^+\mathcal H'_i:= \partial^+\mathcal H_i\cap 
\left\{z: |z\cdot R(e_j)|<\tilde L_1,\, \ {\rm for}\ {\rm all}\ 2\le j\le d\right\} 
\end{gather}
and for $\beta\in (0,1)$ define
 \begin{gather}
   \nonumber 
\mathcal H_{i,\beta}:=\left\{x\in \mathbb Z^d:\, \exists x'\in \mathbb 
  Z^d, |x-x'|_1=1\, \left(x\cdot \ell'-i(L_0+1+L_0^\beta)\right) 
  \right.\times\\
\nonumber
\left.\left(x'\cdot \ell'-i(L_0+1+L_0^\beta)\leq 0\right)\right\}\cap 
\left\{z: |z\cdot R(e_j)|<\tilde L_1,\, \ {\rm for}\ {\rm all}\
  2\le j\le d\right\}. 
\end{gather}
Keeping in mind the above remark   let
\begin{equation}
\label{defctilc}
\tilde c:=c_1 \log\left(\frac{1}{\kappa}\right)\\
\end{equation}
and the asymmetric slab

\begin{equation}
  \nonumber
 U_{\beta,L_0}:=\{x\in \mathbb R^d: x\cdot \ell'\in (-L^\beta_0, L_0)\}.
\end{equation}
We can now define the {\it typical quenched exit event} as

\begin{gather}
\nonumber 
\mathfrak T:=\left\{\omega \in \Omega:\inf_{\substack{z\in \mathcal
      H_{i, \beta}\\ -N\le i\le N+2}}P_{z,\omega}\left[\left(X_{T_{z+U_{\beta, L_0}}}-z\right)\cdot \ell'>0\right]>e^{-\tilde cL_0^\beta },\right.\\
\label{typicalevent}
\left.\inf_{\substack{z\in \mathcal H'_i\\ -N\le i\le N+2}}
  P_{z,\omega}\left[\left(X_{T_{z+U_{\beta, L_0}}}-z\right)\cdot \ell'>0\right]>e^{-\tilde{c}L_0^\beta}\right\}. 
\end{gather}
\medskip

\noindent {\it Step 3: comparing $\hat q$ with $\tilde q$.}
Here we will prove that whenever $\omega\in\mathfrak T$, for all $i\in\mathbb Z$ and $x\in\mathcal H_i$,

\begin{equation}
\label{eqcom}
\hat q(x,\omega)\le e^{2\tilde c L^\beta}\sup_{y\in\mathcal H'_i} \tilde q(y,\omega).
\end{equation}
To prove (\ref{eqcom}) note that on
the event $\{ X_{V_1}\in \mathcal H_{I(x)-1},\, V_1\leq \tilde T \}$,
the  number of excursions from the set $\mathcal H_i$ to its frontal
boundary $\partial^+\mathcal H_i$ before the time $V_1$, 
\begin{equation}\label{numberexrighti}
\mathfrak E_i:=\sum_{n=0}^{V_1-1}\mathds{1}_{\left\{X_{n-1}\in\mathcal H_i,\,  X_n\in \partial ^+\mathcal H_i\right\}}
\end{equation}
is $P_{x,\omega}$-a.s. finite. Before using the finiteness of the
random variables
defined in (\ref{numberexrighti}), we define $U_0:=0$ and  sequences
of $(\mathcal F_n)_{n\geq 0}$-stopping times corresponding
to the moments of the consecutive excursions defined above, as
\begin{gather}
\nonumber 
U_1:=\inf\{n\geq0: X_n\in \partial^+ \mathcal H_i\},\, W_1:=H_{\mathcal H_i}\circ \theta_{U_1}+U_1,\\
\nonumber 
\mbox{and by recursion in }k\geq 1 \mbox{define}\\
\nonumber 
U_{k+1}:=U_1\circ \theta_{W_k}+W_k,\, W_{k+1}:=H_{\mathcal H_i}\circ \theta_{U_{k+1}}+U_{k+1}. 
\end{gather}
Then, for each $i\in\mathbb Z$ and $x\in \mathcal H'_i$, we have 
\begin{gather}
\nonumber 
P_{x,\omega}\left[X_{V_1}\in \mathcal H_{I(x)-1}, V_1\leq \tilde T\right]=\sum_{j=0}^\infty P_{x,\omega}\left[X_{V_1}\in \mathcal H_{I(x)-1}, V_1\leq \tilde T,\mathfrak E_i=j\right]\leq\\
\label{decomleftexit}
P_{x,\omega}\left[X_{V_1}\in \mathcal H_{I(x)-1}, H_{\partial^+\mathcal H_i}=\infty\right]+\sum_{j=1}^\infty P_{x,\omega}\left[X_{V_1}\in \mathcal H_{I(x)-1}, V_1\leq \tilde T,\mathfrak E_i=j\right]. 
\end{gather}
On the other hand, for each $j\geq 1$ we can use successively the strong Markov property to see that 
\begin{gather}
\nonumber 
P_{x,\omega}\left[X_{V_1}\in \mathcal H_{I(x)-1}, V_1\leq \tilde T,\mathfrak E_i=j\right]\leq \\
\nonumber 
P_{x,\omega}\left[U_1<\tilde T \wedge \mathcal H_{I(x)-1},\, P_{X_{U_1},\omega}\left[W_1< \tilde T\wedge H_{\mathcal H_{I(x)+1}},\ldots\right.\right.\\
\nonumber
\left.\left.\ldots ,P_{X_{W_{j-1}},\omega}\left[U_1<\tilde T\wedge H_{\mathcal H_{I(x)-1}},\, P_{X_{U_j},\omega}\left[W_1<\tilde T\wedge \mathcal H_{I(x)+1},\, \right.\right.\right.\right.\\
\label{mchainargu}
\left.P_{X_{W_j},\omega}\left[X_{V_1}\in \mathcal H_{I(x)-1}, V_1\leq \tilde T, H_{\partial^+ \mathcal H_i}=\infty\right] \cdots \right]. 
\end{gather}
Moreover, notice that the last expression in (\ref{mchainargu}) is less than or equal to
\begin{gather}
\label{boundmchain}
\sup_{x\in \mathcal H'_i}P_{x,\omega}\left[X_{V_1}\in \mathcal H_{I(x)-1}, H_{\partial^+ \mathcal H_i}=\infty\right]
\sup_{\substack{y\in \partial^+\mathcal H'_i}}P_{y,\omega}\left[W_1<\tilde T\wedge H_{\mathcal H_{i+1}}\right]^j. 
\end{gather}
Therefore,  applying (\ref{boundmchain}) to inequality
(\ref{decomleftexit}) we get that 
\begin{gather}
\nonumber 
P_{x,\omega}\left[X_{V_1}\in \mathcal H_{I(x)-1}, V_1\leq \tilde T\right]\leq \sup_{x\in \mathcal H'_i}P_{x,\omega}\left[X_{V_1}\in \mathcal H_{I(x)-1}, H_{\partial^+ \mathcal H_i}=\infty\right] \\
\label{decomlexit1}
\times\sum_{j= 0}^\infty\sup_{\substack{y\in \partial^+\mathcal H'_i}}P_{y,\omega}\left[W_1<\tilde T\wedge H_{\mathcal H_{i+1}}\right]^j. 
\end{gather}
Note that 
 for each $-N\le i\le N$ and each $y\in \partial^+ \mathcal H'_i$,
 we know that there is another point $y'\in H_{i,\beta}$ which can be joined to $y$ using a nearest neighbour path inside of box $B_1$ of length less than
or equal to $c_1L_0^\beta$ [cf. (\ref{c1})]. Thus, since $\omega\in\mathfrak T$, 
\begin{gather*}
P_{y,\omega}\left[W_1<\tilde T\wedge H_{\mathcal H_{i+1}}\right]\leq 
1-P_{y,\omega}\left[W_1\geq H_{\mathcal H_{i+1}}\right]\\
\leq 1-\kappa^{c_1L_0^\beta}P_{y',\omega}\left[\left(X_{T_{y'+U_{\beta,L_0}}}-y'\right)\cdot\ell'>0 \right]\\
\stackrel{(\ref{typicalevent})}<1-\kappa^{-c_1L_0^\beta}e^{-\tilde cL_0^\beta}\stackrel{(\ref{defctilc})}=1-e^{-2\tilde cL_0^\beta}. 
\end{gather*}
Using the fact that $\sum_{j=0}^\infty (1-e^{-2\widetilde cL_0^\beta})^j=
e^{2\widetilde cL_0^\beta}$, this finishes the proof
of (\ref{eqcom}).

\medskip

\noindent {\it Step 4: bound on the atypical exit probability in
 the typical quenched exit event.}
Here we will prove that
\begin{gather}
\label{productilqi2}
\mathbb E\left[P_{0,\omega}\left[\tilde T_{-L_1}^{\ell'}\leq T_{L_1}^{\ell'} \wedge \tilde T\right],\, \mathfrak T\right] 
\le
\sum_{m=0}^{N+1}\prod_{-N<i\leq m}\left(e^{3\tilde cL_0^\beta}\mathbb E\left[\tilde q(i,\omega)\right]\right). 
\end{gather}
Notice that by inequality (\ref{quenine}), we find that 
\begin{gather}
\label{eqcomb}
P_{0,\omega}\left[\tilde T_{-L_1}^{\ell'}\leq T_{L_1}^{\ell'} \wedge \tilde T \right]\leq 
\frac{\sum_{m=0}^{ N+1} \prod_{m<i\leq N+1}\rho_i^{-1}}{\prod_{-N< j\leq N+1}\rho_j^{-1}}=
\sum_{m=0}^{ N+1}\prod_{-N<i\leq m}\rho_i. 
\end{gather}
Now, on the  typical quenched exit event $\mathfrak T$ we have that
\begin{gather}
\label{eqcoma}
\sum_{m=0}^{ N+1}\prod_{-N<i\leq m}\rho_i\le\sum_{m=0}^{ N+1}\prod_{-N<i\leq m}\left(e^{\tilde c L_0^\beta}\, \hat q (i,\omega) 
\right). 
\end{gather}
Combining the bounds (\ref{eqcom}) with (\ref{eqcoma}) and (\ref{eqcomb}), we conclude that

\begin{gather}
\label{productilqi}
\mathbb E\left[P_{0,\omega}\left[\tilde T_{-L_1}^{\ell'}\leq T_{L_1}^{\ell'} \wedge \tilde T\right],\, \mathfrak T\right] 
\le\sum_{ m=0}^{ N+1}\mathbb E\left[\prod_{-N<i\leq m}\left(e^{3\tilde cL_0^\beta}\tilde q(i,\omega)\right)\right]. 
\end{gather}
Now we use the crucial observation that as $-N\le i\le N$,
 the random variables $\tilde q(i,\omega)$ are independent,
which finishes the proof of (\ref{productilqi2}).

\medskip
\noindent {\it Step 5: refined bound on the atypical exit probability
in the typical quenched exit event.} Here we will refine the
bound (\ref{productilqi2}), showing that there exists
a constant $c_3$, such that

\begin{gather}
\label{finestileftexit}
\mathbb E\left[P_{0,\omega}\left[\tilde T_{-L_1}^{\ell'}\leq \tilde T\wedge T_{L_1}^{\ell'}\right],\mathfrak T\right]\leq 
\sum_{m=0}^{ N+1}\left(c_3 e^{3\tilde cL_0^\beta}\tilde L_1^{d-1} \kappa^{-3c_1}\mathbb E\left[ q_{{\bf S}_0}\right]\right)^{N+m}
\end{gather}
 Observe that for each $i\in\mathbb Z$ and  $y\in \mathcal H'_i$, there exist a point
 $y'\in\{z:|z\cdot R(e_k)|<\tilde L_1\}\cap \mathbb Z^d $ such that 
$|y+3\ell'-y'|_1\le 1$ and a self-avoiding nearest neighbour path of length at most $3c_1$ connecting $y$ with $y'$.
Therefore,

\begin{gather*}
q_{{\bf S}_0}(\omega)\circ \theta_{y'}=P_{y',\omega}\left[X_{T_{B_{{\bf S}_0+y'}}}\notin \partial^+B_{{\bf S}_0+y'}\right] \\
\ge\kappa^{3c_1}P_{y,\omega}\left[X_{V_1}\in \mathcal H_{I(y)-1}, \,H_{\partial^+\mathcal H_i}=\infty\right]. 
\end{gather*}
Hence defining 
\begin{gather*}
\mathcal H_{i,3}:=\left\{z\in \mathbb Z^d:\, \exists y\in \mathbb Z^d
  \,|z-y|_1=1, (z-iL_0-3)(y-iL_0-3)\leq 0 \right\}\\
\cap\{z:|z\cdot R(e_k)|<\tilde L_1,\ {\rm for}\ {\rm all}\ 2\le k\le d\}
\end{gather*}
we see that
\begin{equation*}
\tilde q(i,\omega)\leq \kappa^{-3c_1}\sup_{\substack{z\in \mathcal H_{i,3}}}P_{z,\omega}\left[X_{T_{B_{{\bf S}_0+z}}}\notin \partial^+B_{{\bf S}_0+z}\right]=\kappa^{-3c_1}
\sup_{\substack{z\in \mathcal H_{i,3}}}q_{{\bf S}_0}(\omega)\circ \theta_{z}.
\end{equation*}
Using the bound $\sup_{\substack{z\in \mathcal H_{i,3}}}q_{{\bf S}_0}(\omega)\circ \theta_{z}\le \sum_{z\in \mathcal H_{i,3}}q_{{\bf S}_0}(\omega)\circ \theta_{z}$,
we finish the proof of inequality
(\ref{finestileftexit}).

\medskip
\noindent {\it Step 6: bound  on
the atypical exit event.}
 Here we will show that  there exists a constant $\mu>0$ such that
\begin{equation}
\label{atypannest}
\mathbb P\left[\mathfrak T^c\right]\leq c_4N\tilde L_1^{d-1}e^{-\mu L_0^{d(2\beta-1)}}
\end{equation}
holds, for some suitable constant $c_4$. Indeed,
by Theorem 4.4 of \cite{Sz04} (see also \cite{Sz01,Sz02}), we know that 
 there
exist constants $\mu>0$ (not depending on $\ell'\in O$) and $L'$ such that
for all $L\ge L'$ one has that

\begin{equation}
\label{aqe}
\mathbb P\left[P_{0,\omega}\left[X_{T_{U_{\beta, L}}}\cdot \ell' >0\right]\leq e^{-\tilde cL^\beta}\right]\leq e^{-\mu L^{d(2\beta-1)}}.
\end{equation}
Choosing $L_0\ge L'$, using the bound (\ref{aqe}) for all the points which are
in some $\mathcal H'_i$ or $\mathcal H'_{i,\beta}$ for
some $N\le i\le N+2$, and the fact that the cardinality
of these points is $c_4N\tilde L_1^{d-1}$ for some
constant $c_4$, we obtain (\ref{atypannest}).

\medskip

\noindent {\it Step 7: upper bound on the lateral exit probability.}
Using exactly the same argument as the one presented
in pages 524-526 of the proof of Proposition 2.1 of \cite{Sz02},
 we see that there is a constant $c_5$ such that

\begin{equation}
\label{finalboundlat}
P_0\left[\tilde T\leq \tilde T_{-L_1}^{\ell'}\wedge T_{L_1}^{\ell'}\right]\leq (2d-2)\left(c_5\tilde L_1^{(d-2)}\frac{L_1^3}{L_0^2}\tilde L_0 \mathbb 
E\left[q_{{\bf S}_0}(\omega)\right]\right)^{\frac{\tilde N}{12N}},
\end{equation}
where we used that  $\tilde N\geq 48 N$ and that $N\geq 3$.

\medskip

\noindent {\it Step 8: conclusion.} 
In view of (\ref{finestileftexit})-(\ref{atypannest}) and (\ref{finalboundlat}),we see that 
there exist positive constants $c_3, c_4, c_5$ and $c_6$ such that 
\begin{gather}
\nonumber 
\mathbb E\left[q_{{\bf S}_1}(\omega)\right]\leq P_0\left[
\tilde T_{-L_1}^{\ell'}\leq \tilde T\wedge T_{L_1}^{\ell'}\right]+
P_0\left[\tilde T\leq \tilde T_{-L_1}^{\ell'} \wedge T_{L_1}^{\ell'}\right]\\
\nonumber 
\leq \mathbb E\left[P_{0,\omega}\left[\tilde T_{-L_1}^{\ell'}\leq 
\tilde T\wedge T_{L_1}^{\ell'}\right],\, \mathfrak T\right]+ \mathbb P\left[ \mathfrak T^c\right]+P_0\left[\tilde T\leq \tilde T_{-L_1}^{\ell'} \wedge T_{L_1}^{\ell'}\right]\\
\nonumber 
\stackrel{(\ref{finestileftexit})-(\ref{atypannest})-(\ref{finalboundlat})}\leq \sum_{m=0}^{ N+1}\left(c_3\kappa^{-3c_1}\tilde L_1^{d-1}e^{3\tilde cL_0^\beta} \mathbb E\left[ q_{{\bf S}_0}(\omega)\right]\right)^{m+N}\\
\label{finalboundprop}
+c_4 N\tilde L_1^{d-1}e^{-\mu L_0^{d(2\beta-1)} }+ c_6\left(c_5\tilde L_1^{d-2)}\frac{L_1^3}{L_0^2}\tilde L_0 \mathbb E\left[q_{{\bf S}_0}(\omega)\right]\right)^{\frac{\tilde N}{12N}}. 
\end{gather}
\end{proof}

\subsection{Recursion}
\label{sectionrecursion}
Here we will use the result of Proposition \ref{prop1}, to inductively derive
a bound   for the exit probability
through atypical sides of boxes at an increasing sequence of scales.
Since we assume $(T')|\ell$, we know that
there is an open set $O\subset\mathbb S^{d-1}$ containing $\ell$,
such that for $\ell'\in O$, the atypical exit probabilities from slabs decay like
stretched exponentials.
Now, given  $\ell'\in O$,
we choose a rotation $R$, with $R(e_1)=\ell'$.
Let $\nu>0$.
 We next  consider  sequences of scales $\left(L_k\right)_{k\geq0}$ and $\left(\tilde L_k\right)_{k\geq0}$,
defining a sequence of triples $({\bf S}_k)_{k\ge 0}$
associated to the corresponding boxes
$$
 B_k:= B_{{\bf S}_k}:=R\left( (-L_k,L_k)\times (-\tilde L_k,\tilde L_k)^{d-1}
\right)\cap\mathbb Z^d 
$$
according to the notation (\ref{triple}),
that satisfy
\begin{gather}
\label{scalesk0}
L_0>3\sqrt d\\
\label{scalesk}
L_0^3>\tilde L_0>L_0,
\end{gather}
and

\begin{gather}
  \label{scalesk1}
  L_k=\nu L_{k-1},\hspace{0.5ex}\mbox{for }k\geq 1, \\
  \label{scalesk2}
  \tilde L_k=\nu^3 \tilde L_{k-1},\hspace{0.5ex}\mbox{for }k\geq 1.
  \end{gather}
Throughout we write $q_k$ and $B_k$ in place of $q_{{\bf S}_k}$ and
 $B_{{\bf S}_k}$.
 We will also need to use the fact that $(T')|\ell$ implies
 that for $\beta\in (3/4,1)$,
 there is a constant $c_7(d,\kappa)$ and  an $L''>0$
such that whenever we choose $L_0\ge L''$ one has that

\begin{equation}
\label{tgammaexp}
\mathbb E\left[q_0\right]\le e^{-c_7 L_0^{\frac{\beta+1}{2}}}.
\end{equation}
 The next lemma will be instrumental for the final proof. 

\medskip
\begin{lemma}
\label{lemmarecur}
Let $d\ge 2$. Consider a random walk in random environment
satisfying condition $(T')|l$. Then, there exists
$\nu=\nu(d,\kappa)>0$, $L_0=L_0(d,\kappa)>0$, $\tilde L_0=\tilde L_0(d,\kappa)>0$
satisfying (\ref{scalesk0}) and (\ref{scalesk}), and
a constant $c_{8}(d,\kappa)>0$  such that for all $\ell'\in O$ and    $k\ge 0$
we have that for $(L_k)_{k\ge 1}$ defined as in (\ref{scalesk1}) and
$(\tilde L_k)_{k\ge 1}$  as in (\ref{scalesk2}), 

\begin{equation}
\label{recurtionphik}
\mathbb E[q_k]\leq e^{- c_{8} L_k}.
\end{equation}
\end{lemma}
\begin{proof}
Note that by (\ref{tgammaexp}), we have that

\begin{equation}
  \label{induction0}
\mathbb E\left[q_0\right]\le e^{-d_{0} L_0},
\end{equation}
with

$$
d_{0}:=\frac{c_7}{L_0^{\frac{1-\beta}{2}}},
$$
for some fixed $\beta\in (3/4,1)$. 
Let us now define recursively for $k\ge 0$,

$$
d_{k+1}:=d_k-\left(\left(1+3c_1\log\frac{1}{\kappa}\right)L_0^\beta+3\right)\frac{1}{\nu^{(1-\beta)k}}.
  $$
  We will first prove by induction on $k\ge 0$, that

  \begin{equation}
    \label{inductionk}
  \mathbb E\left[q_k\right]\le e^{-d_{k} L_k}.
\end{equation}
Note that (\ref{inductionk}) is satisfied for $k=0$ (which is (\ref{induction0})).
 Let us now assume that (\ref{inductionk}) is satisfied for $k\ge 0$.
We will prove that it is then also satisfied by $k+1$.
Note that $L_k=\nu^k L_0$ while $\tilde L_k=\nu^{3k} \tilde L_0$.
Now, by  inequality (\ref{inepropo}) of Proposition \ref{prop1}
we have that

\begin{gather}
\nonumber 
\mathbb E[q_{k+1}]\leq 
(\nu+2) 
\left(c_2\kappa^{-3c_1}\tilde L_{k+1}^{d-1}e^{3c_1(\log(1/\kappa)) L_k^\beta}
\right)^{2\nu+2}\mathbb E\left[q_{k}\right]^\nu\\
\nonumber
+
\left(c_2\tilde L_{k+1}^{d-2}\frac{L_{k+1}^3}{L_k^2}\tilde L_k\mathbb E\left[q_{k}\right]\right)^{\frac{\nu^3}{12 \nu}}
+c_2\nu\tilde L_{k+1}^{d-1}e^{-\mu L_k^{d(2\beta-1)}}\\
\nonumber
=
(\nu+2)
\left(c_2\kappa^{-3c_1}\nu^{3(k+1)(d-1)} \tilde
  L_0^{d-1}e^{3c_1(\log(1/\kappa))\nu^{\beta k}L_0^\beta}
\right)^{2\nu+2}\mathbb E\left[q_k\right]^\nu\\
  \label{eqlemma1}
+
\left(c_2\nu^{3(d-2)(k+1)-1} \tilde L_0^{d-1} L_0\mathbb 
E\left[q_k\right]\right)^{\frac{\nu^2}{12}}
+c_2\nu^{3(d-1)(k+1)}\tilde L_0^{d-1}e^{-\mu
  \nu^{dk(2\beta-1)}L_0^{d(2\beta-1)}}.
\end{gather}
We will analyze each of the terms of (\ref{eqlemma1}) separately. Note
that the third term can be written as

$$
\exp\left\{-\mu\nu^{kd(2\beta-1)}L_0^{d(2\beta-1)}+\nu^kL_0
  d_0+3(d-1)(k+1)\nu+\log(c_2\tilde L_0^{d-1})\right\} e^{-d_{0}L_{k+1}}.
$$
Now, the exponent of the first exponential of the above expression is
bounded from above by

$$
-(\mu\nu^{d(2\beta-1)-1}L_0^{d(2\beta-1)}-L_0d_{0})\nu.
$$
Hence, using the fact that for $\beta\in (3/4,1)$, one has that
$d(2\beta-1)>1$, we can see that there is a $\nu_0=\nu_0(d,\kappa)$ (also depending on the choice
of $L_0$)
 such that for 
$\nu\ge\nu_0$, the above expression is bounded from above by 
$\log\frac{1}{3}$ and hence the third term of (\ref{eqlemma1}) is 
bounded from above by 

\begin{equation}
  \label{common}
\frac{1}{3}e^{-d_{0}L_{k+1}}. 
\end{equation}
A similar analysis lets us conclude that there is a
$\nu_1=\nu_1(d,\kappa)>\nu_0$ (also depending on the choice
of $L_0$) such that the second term of the
right-hand
side of (\ref{eqlemma1}) is bounded from above also by (\ref{common}).
Let us now write the first term of the right-hand side of (\ref{eqlemma1}) as 

\begin{gather}
 \nonumber 
\exp\left\{\left(\log\left((\nu+2)\nu^{3(k+1)(d-1)}\right)+\log\frac{c_2\tilde 
    L_0^{d-1}}{\kappa^{3c_1}}\right.\right.\\
\nonumber 
+\left.\left.3c_1\left(\log\frac{1}{\kappa}\right) L_0^\beta\nu^{\beta k}\right)(2\nu+2)-d_k\nu^{k+1}L_0\right\}. 
\end{gather}
Now, note that there is a $\nu_2=\nu_2(d,\kappa)\ge \nu_1$
(also depending on the choice
of $L_0$) such that for 
$\nu\ge \nu_2$, the above expression is bounded from above by 

$$
\left(\left(\left(1+3c_1\log\frac{1}{\kappa}\right) L_0^\beta+1\right)\frac{1}{\nu^{(1-\beta)k}}-d_k\right)L_{k+1}
$$
which proves that the first term of the right-hand side of
(\ref{eqlemma1}) is bounded from above by

$$
\frac{1}{3}e^{-d_{k+1}L_{k+1}}.
$$
Combining this estimate with (\ref{common}), we see that
(\ref{inductionk}) is satisfied for $k+1$. Finally,
note that

$$
d_k\ge d_0-\sum_{k=1}^\infty \left(\left(1+3c_1\log\frac{1}{\kappa}\right)L_0^\beta+3\right)\frac{1}{\nu^{(1-\beta)k}}=:c_{8}>0,
$$
where the last inequality is satisfied whenever $\nu\ge \nu_3$
for some $\nu_3(d,\kappa)$ (also depending on the choice
of $L_0$).
\end{proof}

\subsection{Final step in the  proof of Theorem \ref{theoremequivalence}}
\label{sectionfinal} 
The same argument as the one presented in the proof of Proposition 2.3
leading to (2.57) of \cite{Sz02}, shows that (\ref{recurtionphik}) of Lemma
\ref{lemmarecur}, implies  that there is an open set $O\subset\mathbb
S^{d-1}$ such that for all $\ell'\in O$ one has that

\begin{gather*}
\lim_{L\to\infty}\frac{1}{L}\log P_0\left[\tilde T_{- L}^{\ell'}<T_{L}^{\ell'}\right]<0.
\end{gather*}

\end{document}